\documentclass[11pt]{article}

\usepackage{enumerate}
\usepackage{amsmath}
\usepackage{amssymb,latexsym}
\usepackage{amsthm}
\usepackage{color}
\usepackage{cancel}
\usepackage{graphicx}

\usepackage{tikz}

\usepackage{hyperref}

\usepackage{amscd}

\DeclareMathOperator{\C}{\mathcal{C}}

\newtheorem{theorem}{Theorem}[section]

\newtheorem{lemma}[theorem]{Lemma}
\newtheorem{corollary}[theorem]{Corollary}
\newtheorem{definition}[theorem]{Definition}
\newtheorem{proposition}[theorem]{Proposition}

\newtheorem{example}[theorem]{Example}

\newtheorem{remark}[theorem]{Remark}
\newtheorem{notation}[theorem]{Notation}

\newcommand{\fqn}{\mathbb{F}_{q^n}}
\newcommand{\fqk}{\mathbb{F}_{q^k}}

\newcommand{\cL}{{\mathcal L}}

\newcommand{\F}{{\mathbb F}}

\newcommand{\fq}{{\mathbb F}_{q}}

\newcommand{\lmb}{\lambda}

\newcommand{\N}{\mathrm{N}}

\title{Constructions and equivalence of Sidon spaces}

\author{Chiara Castello, Olga Polverino, Paolo Santonastaso and 
Ferdinando Zullo\thanks{The research was supported by the project ``VALERE: VAnviteLli pEr la RicErca" of the University of Campania ``Luigi Vanvitelli'' and was partially supported by the Italian National Group for Algebraic and Geometric Structures and their Applications (GNSAGA - INdAM).}}

\date{}

\begin{document}
\maketitle

\begin{abstract}
Sidon spaces have been introduced by Bachoc, Serra and Z\'emor in 2017 as the $q$-analogue of Sidon sets. The interest on Sidon spaces has increased quickly, especially after the work of Roth, Raviv and Tamo in 2018, in which they highlighted the correspondence between Sidon spaces and cyclic subspace codes.  
Up to now, the known constructions of Sidon Spaces may be divided in three families: the ones contained in the sum of two multiplicative cosets of a fixed subfield of $\mathbb{F}_{q^n}$, the ones contained in the sum of more than two multiplicative cosets of a fixed subfield of $\mathbb{F}_{q^n}$ and finally the ones obtained as the kernel of subspace polynomials. 
In this paper we will mainly focus on the first class of examples, for which we provide characterization results and we will show some new examples, arising also from some well-known combinatorial objects.
Moreover, we will give a quite natural definition of equivalence among Sidon spaces, which relies on the notion of equivalence of cyclic subspace codes and we will discuss about the equivalence of the known examples.
\end{abstract}

{\textbf{Keywords}: Sidon space; cyclic subspace code; linearized polynomial; $q$-analog}\\

{\textbf{MSC2020}: Primary 11T99; 11T06. Secondary   11T71; 94B05.}

\section{Introduction}

Bachoc, Serra and Z\'emor in \cite{bachoc2017analogue} introduced the notion of Sidon space as an important object to study in order to prove the linear analogue of Vosper's Theorem, which characterizes the equality in the linear analogue of Cauchy-Davenport inequality proved in \cite{bachoc2018revisiting,hou2002generalization}.
Sidon spaces can be seen as the $q$-analog of Sidon sets, a well-studied combinatorial object introduced by Simon Szidon, see \cite{erdHos1981solved}.
The definition of Sidon space is the following.
An $\fq$-subspace $V$ of $\fqn$ is called a \textbf{Sidon space} if for all nonzero $a,b,c,d \in V$ such that $ab=cd$ then 
\[\{a \F_q,b \F_q\}=\{c \F_q,d \F_q\},\]
where $e\fq =\{e\lambda \colon \lambda \in \fq\}$.

Sidon spaces gained a lot of attention especially because of their connection with cyclic subspace codes, pointed out by Roth, Raviv and Tamo in \cite{roth2017construction}. Indeed, this connection allowed them to prove, for most of the cases, a conjecture on the existence of cyclic subspace codes by Horlemann-Trautmann, Manganiello, Braun and Rosenthal in \cite{trautmann2013cyclic}.

Let $n$ and $k$ be positive integers such that $k \leq n$. We denote by $\mathcal{G}_q(n,k)$ the set of $k$-dimensional $\F_q$-subspaces of $\F_{q^n}$. For $U \in \mathcal{G}_q(n,k)$ and $\alpha \in \fqn$, denote by $\mathrm{Orb}(U)=\{\alpha U \colon \alpha \in \fqn^*\}$ the \textbf{orbit} of $U$, where $\alpha U=\{\alpha u \colon u \in U\}$. 
We can equip $\mathcal{G}_q(n,k)$ with the \textbf{subspace metric}, that is 
\[d(U,V)=\dim_{\F_q}(U)+\dim_{\F_q}(V)-2\dim_{\F_q}(U \cap V),\]
where $U,V \in \mathcal{G}_{q}(n,k)$.
A \textbf{constant dimension subspace code} as a subset $\mathcal{C}$ of $\mathcal{G}_{q}(n,k)$ endowed with the subspace metric.
In particular, its minimum distance $d(\C)$ is the minimum of distances between two distinct elements in $\C$.
Codes in the subspace metric have been investigated especially after the well-celebrated paper of Koetter and Kschischang \cite{koetter2008coding}, in which they showed how to use such codes in random network coding.
Because of their algebraic structure, the most investigated class is those of \textbf{cyclic subspace codes}, originally introduced by Etzion and Vardy in \cite{etzion2011error}.
A constant dimension subspace code $\mathcal{C} \subseteq \mathcal{G}_q(n,k)$ is said to be \textbf{cyclic} if for every $\alpha \in \F_{q^n}^*$ and every $V \in C$ then $\alpha V \in C$, that is $\mathcal{C}$ is the union of orbit of subspaces in $\mathcal{G}_q(n,k)$.

The following result establishes a correspondence between Sidon spaces and cyclic subspace codes.

\begin{theorem}\cite[Lemma 34]{roth2017construction}\label{lem:charSidon}
Let $U \in \mathcal{G}_q(n,k)$.
Then $U$ is a Sidon space if and only if  $|\mathrm{Orb}(U)|=\frac{q^n-1}{q-1}$ and $d(\mathrm{Orb}(U))=2k-2$, i.e. for all $\alpha \in \mathbb{F}_{q^n}\setminus \fq$
\[\dim_{\fq}(U\cap \alpha U)\leq 1.\]
\end{theorem}

Note that $|\mathrm{Orb}(U)|=\frac{q^n-1}{q-1}$ is equivalent to require that $U$ is strictly $\fq$-linear; see \cite[Theorem 1]{otal2017cyclic}.

Therefore, equivalently we can talk about $k$-dimensional Sidon spaces or optimal one-orbit cyclic subspace codes with minimum distance $2k-2$.

In this paper we will use the terminology of Sidon spaces and our aim is to show new constructions of such spaces and to give some insight on the equivalence of these objects. All of our results may be read in terms of cyclic subspace codes via the correspondence given in Theorem \ref{lem:charSidon}.

Known examples of Sidon spaces can be divided in three families, based on their descriptions:
\begin{enumerate}
    \item as subspace of the sum of two multiplicative cosets\footnote{In this paper, a multiplicative coset of a subfield $\fqk$ of $\fqn$ is $\alpha\fqk$, for some $\alpha \in \fqn$.} of a fixed subfield of $\mathbb{F}_{q^n}$;
    \item as subspace of the sum of more than two multiplicative cosets of a fixed subfield of $\mathbb{F}_{q^n}$;
    \item as kernel of subspace polynomials.
\end{enumerate}

In the first family we find the examples in \cite[Constructions 11 and 15]{roth2017construction}, \cite[Lemma 3.1]{zhang2023constructions}, \cite[Lemmas 4.1 and 4.2]{li2021cyclic} and \cite[Lemma 2.4]{zhang2023further}, in the second one we find examples in \cite[Theorems 3.2 and 3.6]{zhang2022new}, \cite[Lemma 3.3, Theorems 3.5 and 3.13]{zhang2023constructions}, \cite[Theorems 4.5, 4.7, 4.9, 4.13-18 and 4.20]{li2021cyclic} and \cite[Lemma 5 and Theorem 6]{zhang2022constructions}, and for the last family we can find examples in \cite[Corollary 4 and Section IIIA]{ben2016subspace} and \cite[Section 7.3]{santonastaso2022linearized}.

We will deal with Sidon spaces which belong to the first family, that is $k$-dimensional Sidon spaces of $\mathbb{F}_{q^{k\ell}}$ contained in $\mathbb{F}_{q^k}+\gamma\mathbb{F}_{q^k}$, for some $\gamma \in \mathbb{F}_{q^{k\ell}}\setminus\fqk$.
We will first give a characterization result and we will also obtain a description with the aid of linearized polynomials. As a matter of fact, we will show how to construct examples from \emph{scattered polynomials}. We will also show examples of linearized polynomials defining Sidon spaces, which are not scattered. 
Moreover, we will also investigate the direct sum of Sidon spaces, already exploited in \cite{li2021cyclic}, in order to construct Sidon spaces in higher degree field extensions with less assumptions. This will give new examples of Sidon spaces.
Finally, we will deal with the notion of equivalence of Sidon spaces, following the approach used in \cite{zullo2023multi}.
Indeed, since many more examples of Sidon spaces are emerging, it is important to give a definition of equivalence for them, which can naturally arise from the classification of isometries for cyclic subspace codes established by Gluesing-Luerssen and Lehmann in \cite{gluesing2021automorphism}. 
Thanks to the above equivalence, we can see how all the known examples of the first family arise from a particular scattered polynomial. As a consequence, we will show that the family of examples we found is very large and contains many non equivalent examples.
We will then prove that some of our examples cannot be obtained as the kernel of a subspace trinomial and hence they cannot be obtained from the known examples of the third family.

\section{Linearized polynomials and scattered polynomials}

A \textbf{linearized polynomial} (or $q$-\textbf{polynomial}) is a polynomial of the shape
$$ f(x)=\sum_{i=0}^{t}f_i x^{q^i}, \quad f_i \in \F_{q^k}.$$
If $f$ is nonzero, the $q$-\textbf{degree} of $f$ will be the maximum $i$ such that $f_i \neq 0$.

The set of linearized polynomials forms an $\F_q$-algebra with the usual addition between polynomials and the composition, defined by
$$ (f_ix^{q^i})\circ (g_jx^{q^j})=f_ig_j^{q^i}x^{q^{i+j}}, $$
on $q$-monomials, and then extended by associativity and distributivity and the multiplication by elements of $\F_q$. We denote this $\F_q$-algebra by $\cL_{k,q}$.
The elements of the quotient algebra $\tilde{\mathcal{L}}_{k,q}$ obtained from $\mathcal{L}_{k,q}$  modulo the two-sided ideal $(x^{q^k}-x)$ are represented by linearized polynomials of $q$-degree less than $k$, i.e.
$$ \tilde{\mathcal{L}}_{k,q} = \left\{\sum_{i=0}^{k-1}f_i x^{q^i}, \quad f_i \in \F_{q^k} \right\}.$$
It is well-known that the following isomorphism as $\F_q$-algebra holds
\[ (\tilde{\cL}_{k,q},+,\circ)\cong(\mathrm{End}_{\fq}(\F_{q^k}),+,\circ),
\] 
where the linearized polynomial $f$ is identified with the endomorphism of $\F_{q^k}$ defined by its evaluation map.

Thanks to the above isomorphism, we can identify a linearized polynomial with the endomorphism it defines and so we will speak of \emph{kernel} and \emph{rank} of a $q$-polynomial $f$ meaning by this the kernel and rank of the corresponding endomorphism, denoted by $\ker(f)$ and $\mathrm{rk}(f)$, respectively. 

Consider
\[f(x)=a_0x+a_1x^{q^s}+\cdots+a_{k-1}x^{q^{s(k-1)}}+a_kx^{q^{sk}}\in \mathcal{L}_{n,q},\]
with $k\leq n-1$, $\gcd(s,k)=1$ and let $a_k\ne 0$. Then 
\begin{equation}\label{eq:boundnumroots} 
\dim_{\F_q}(\ker (f))\leq k. 
\end{equation}
When the equality holds, we say that $f$ is a \textbf{subspace polynomial}.
It is clear that every $k$-dimensional $\fq$-subspace of $\fqn$ is the kernel of a unique monic subspace polynomial of $q$-degree $k$.

We refer to \cite{lidl1997finite} for more details on this class of polynomials.

In connection with optimal codes in the rank metric, Sheekey in \cite{sheekey2016new} introduced the notion of scattered polynomials.

\begin{definition}
A polynomial $f \in {\mathcal{L}}_{k,q}$ is said to be \textbf{scattered} if for any $a,b \in \mathbb{F}_{q^k}^*$ such that $f(a)/a=f(b)/b$ implies that $a$ and $b$ are $\fq$-proportional.
\end{definition}

If we denote by $U_f=\{(x,f(x))\colon x \in \F_{q^k}\}\subseteq \mathbb{F}_{q^k}^2$ the $k$-dimensional $\fq$-subspace defined by the graph of $f$, the polynomial $f$ turns out to be scattered if for any $a \in \mathbb{F}_{q^k}^*$ we have
\begin{equation}\label{eq:scattcondSa,f} 
\dim_{\fq}(S_{a,f})= 1, 
\end{equation}
where $S_{a,f}=\{ \rho\in \F_{q^k}\colon \rho(a,f(a)) \in U_f \}$, i.e. $S_{a,f}=\fq$.

\section{Characterization of Sidon spaces}

Let $\gamma\in\fqn$ be a root of an irreducible polynomial of degree $\frac{n}{k}\geqslant 2$ over $\fqk$, i.e. $\fqn=\fqk(\gamma)$ and $[\F_{q^n}:\F_{q^k}]=\frac{n}{k}$.\\
Let $U$ be an $\fq$-subspace of $\fqk^2$ of dimension $m$ over $\fq$ and let define 
\[V_{U,\gamma}=\lbrace u+v\gamma : (u,v)\in U\rbrace \in\mathcal{G}_q(n,m).\]
For any $(u,v)\in U$, define
\[S_{(u,v)}^\gamma=\lbrace \lmb\in\fqn | \lmb(u+v\gamma)\in V_{U,\gamma}\rbrace\subseteq \fqn.\]
It is easy to see that $S_{(u,v)}^\gamma$ is an $\fq$-subspace of $\fqn$ containing $\fq$.

\begin{definition}
Let $U$ be an $\fq$-subspace of $\fqk^2$ of dimension $m$ over $\fq$ and let $\gamma\in\fqn$ be a root of an irreducible polynomial of degree $\frac{n}{k}\geqslant 2$ over $\fqk$. We say that $U$ has the \textbf{Sidon space property with respect to $\gamma$} if
for every $\fq$-linearly independent $(u,v),(u',v')\in U$ we have that \[S_{(u,v)}^\gamma\cap S_{(u',v')}^\gamma=\fq.\]
\end{definition}

Our aim is to characterize the Sidon space property of a subspace of the form $V_{U,\gamma}$ via the property above defined on $U$. In the next examples we see that this property is strongly related also to $\gamma$ and not only to $U$.

\begin{example}\label{ex:q=3noSidon}
Consider the extension $\mathbb{F}_{3^8}/\mathbb{F}_{3^4}$. Let $\mathbb{F}_{3^4}=\mathbb{F}_{3}(g_1)$ and $\mathbb{F}_{3^8}=\mathbb{F}_{3^4}(\gamma)$, where $g_1\in\mathbb{F}_{3^4}$ has $x^4-x^3-1$ as minimal polynomial over $\mathbb{F}_{3}$ and $\gamma \in \mathbb{F}_{3^8}$ has $x^2+g_1^{35}x+g_1$ as minimal polynomial over $\mathbb{F}_{3^4}$. Note that $\mathbb{F}_{3^8}=\mathbb{F}_{3^4}(\gamma)=\mathbb{F}_{3^4}(\gamma^2)$.\\
Denoting by
\[U=\{(x,x^3) \colon x \in \mathbb{F}_{3^4}\},\]
then $V_{U,\gamma}=\lbrace x+x^3\gamma\colon x\in\mathbb{F}_{3^4}\rbrace$ is a Sidon space in $\mathbb{F}_{3^8}$, whereas $V_{U,\gamma^2}=\lbrace x+x^3\gamma^2\colon x\in\mathbb{F}_{3^4}\rbrace$ is not a Sidon space.
\end{example}

We are going to give a characterization of Sidon spaces which are contained in the sum of two multiplicative cosets of $\mathbb{F}_{q^k}$ via the Sidon space property.

\begin{theorem}
\label{caratterizzazioneSidonCoppie}
Let $\gamma\in\fqn$ be a root of an irreducible polynomial of degree $\frac{n}{k}\geqslant 2$ over $\fqk$. Let $U$ be an $\fq$-subspace of $\fqk^2$ of dimension $m$ over $\fq$ 
and let consider $V=V_{U,\gamma}\in\mathcal{G}_q(n,m)$.  
Then $V$ is a Sidon space if and only if $U$ has the Sidon space property with respect to $\gamma$.
\end{theorem}
\begin{proof}
Suppose that $V$ is a Sidon space and, by contradiction, suppose that there exist $(u,v),(u',v')\in U \setminus\{(0,0)\}$ which are $\fq$-linearly independent
and there exists $w\in (S_{(u,v)}^\gamma\cap S_{(u',v')}^\gamma)\setminus\fq$. Since $w\in S_{(u,v)}^\gamma\cap S_{(u',v')}^\gamma$, we have that
\begin{center}
$w (u+v\gamma)=wu+wv\gamma\in V$\\
$w (u'+v'\gamma)=wu'+wv'\gamma\in V$.
\end{center}
Now, let $a=wu+wv\gamma,b=u'+v'\gamma,c=u+v\gamma,d=w u'+w v'\gamma$. Then $a,b,c,d$ are nonzero elements of $V$ and it results
\begin{equation}
\label{sidondefinition}
ab=cd.
\end{equation}
Since $w \notin \fq$ and $(u,v)$ and $(u',v')$ are not $\fq$-proportional, we get $\lbrace a\fq, b\fq\rbrace\neq\lbrace c\fq, d\fq\rbrace$, a contradiction.
Conversely, suppose that for every $(u,v),(u',v')\in U$ such that $(u,v),(u',v')$ are not $\fq$-proportional we have that $S_{(u,v)}^\gamma\cap S_{(u',v')}^\gamma=\fq$ and by contradiction suppose that $V$ is not a Sidon space, i.e. there exist $a=u+v\gamma, b=u'+v'\gamma, c=t+z\gamma, d=t'+z'\gamma \in V$ distinct nonzero elements of $V$ such that
 \[ab=cd\]
and $\lbrace a\fq, b\fq\rbrace\neq\lbrace c\fq, d\fq\rbrace$. As a consequence, we have that
\begin{equation}
\label{lambda}
\frac{a}{c}=\frac{u+v\gamma}{t+z\gamma}=\frac{t'+z'\gamma}{u'+v'\gamma}=\frac{d}{b}=\lmb\in\fqn\setminus\fq
\end{equation}
and
\begin{equation}
\label{mu}
   \frac{a}{d}=\frac{u+v\gamma}{t'+z'\gamma}=\frac{t+z\gamma}{u'+v'\gamma}=\frac{c}{b}=\mu\in\fqn\setminus\fq. 
\end{equation}
Since $u+v\gamma=\lmb(t+z\gamma)\in V$ and $t'+z'\gamma=\lmb (u'+v'\gamma)\in V$, then $\lmb\in (S_{(u',v')}^\gamma\cap S_{(t,z)}^\gamma)\setminus\fq$. 
Note that this is possible, by hypothesis, if and only if $(u',v')$ and $(t,z)$ are $\fq$-proportional, that is there exists $\eta \in \fq^*$ such that $(u',v')=\eta (t,z)$.
Then we have $\frac{c}{b}=\frac{t+z\gamma}{u'+v'\gamma}=\mu=\frac{1}{\eta}\in\fq$ and this is not possible. 
\end{proof}

If the degree $[\F_{q^k}(\gamma):\F_{q^k}]$ is greater than two, then the Sidon space property of $U$ does not depend on $\gamma$ and we can replace the subspaces $S_{(u,v)}^{\gamma}$ of $\fqn$ by $S_{(u,v)}$, which are subspaces of $\fqk$ and are independent from $\gamma$.

\begin{theorem}\label{thm:charwithNOgamma}
Let $U$ be an $\fq$-subspace of $\fqk^2$ 
and let $\gamma$ be a root of an irreducible polynomial over $\fqk$ of degree greater than two.
Then $U$ has the Sidon space property with respect to $\gamma$ if and only if
\begin{equation}\label{eq:withNOgamma}
    S_{(u,v)}\cap S_{(u',v')}=\fq,
\end{equation}
for any $\fq$-linearly independent $(u,v),(u',v')\in U$, where
\[S_{(u,v)}=\{\lambda \in \mathbb{F}_{q^k} \colon \lambda (u,v) \in U\}.\]
\end{theorem}
\begin{proof}
If $U$ has the Sidon space property with respect to $\gamma$, then \eqref{eq:withNOgamma} trivially holds. 
Now, suppose that \eqref{eq:withNOgamma} holds for every $\fq$-linearly independent $(u,v),(u',v')\in U$.
Fix any $(u,v),(u',v')\in U$ in such a way that are $\fq$-linearly independent and suppose by contradiction that there exists  
\[ \mu \in (S_{(u,v)}^\gamma\cap S_{(u',v')}^\gamma)\setminus\mathbb{F}_q, \]
and $\mu \notin\fqk$ because of the assumption of $S_{(u,v)}\cap S_{(u',v')}$.
Then 
\[ \mu (u+\gamma v), \mu (u'+\gamma v') \in V_{U,\gamma}, \]
i.e. 
\begin{equation}\label{eq:muinSuvu'v'} 
\mu (u+\gamma v)=w+\gamma z\,\,\, \text{and}\,\,\, \mu (u'+\gamma v')=w'+\gamma z', 
\end{equation}
for some $(w,z),(w',z') \in U$.
Therefore,
\begin{equation}\label{eq:u+vgammau'+gammav'} 
(u+\gamma v)(w'+\gamma z')=(u'+\gamma v')(w+\gamma z), 
\end{equation}
and hence since $[\F_{q^k}(\gamma):\F_{q^k}]>2$ it follows
\begin{equation}\label{eq:systconduu'vv'}
\left\{
\begin{array}{lll}
uw'-u'w=0,\\
vz'-v'z=0,\\
uz'+vw'=u'z+v'w.
\end{array}
\right.
\end{equation}
Note that since $\mu \notin \fqk$ the pairs $(u,w),(v,z),(u',w')$ and $(v',z')$ are different from $(0,0)$.
We start by proving that $w$ and $w'$ cannot be zero.\\
\textbf{Case 1}: $w=0$.\\
By System \eqref{eq:systconduu'vv'} it follows that $w'=0$ and $u\ne 0$.\\
\textbf{Case 1.1:} $v=0$.\\ 
In this case, also $v'=0$ and so we obtain
\[ \frac{\mu}{\gamma}=\frac{z}u=\frac{z'}{u'} \in \fqk, \]
hence $\frac{\mu}{\gamma} \in S_{(u,0)}\cap S_{(u',0)}$, implying that $\mu/\gamma=\alpha \in \fq$. Also, $(u,v)=(u,0)$, $(u',v')=(u',0)$ $(0,z/\alpha)=(0,u)$ and $(0,z'/\alpha)=(0,u')$ are in $U$. 
Let $\bar{\alpha}=u'/u$, then $\bar{\alpha}\in S_{(u,0)}\cap S_{(0,u)}=\fq$, that is
$(u,v)$ and $(u',v')$ are $\fq$-proportional, a contradiction.\\ \textbf{Case 1.2:} $v \ne 0$.\\
Therefore, by the second and third equations of System \eqref{eq:systconduu'vv'} we get $\beta=u/u'=v/v'=z/z' \in S_{(u',v')}$.
Since $(0,z)$ and $(0,z')$ are both in $U$, then we also obtain that $\beta \in S_{(0,z')}$.
Since $u'\ne 0$ then $\beta \in S_{(u',v')} \cap S_{(0,z')}=\fq$, which implies that $(u,v)$ and $(u',v')$ are $\fq$-proportional.
\newline
\textbf{Case 2:} $w'=0$.\\
Argue as for Case 1.\\
\textbf{Case 3:} $w,w' \ne 0$.\\
We have two subcases that need to be analyzed.\\
\textbf{Case 3.1:} $z,z'\ne 0$.\\ 
Let
\[ \lambda=\frac{u}{w}=\frac{u'}{w'}\in \F_{q^k}\,\,\,\text{and}\,\,\,\rho=\frac{v}{z}=\frac{v'}{z'}\in \F_{q^k}. \]
From the third equation of System \eqref{eq:systconduu'vv'}, we have
\[ wz'\lambda+zw'\rho=w'z\lambda+z' w\rho, \]
that is
\[ (\rho-\lambda) (zw'-z'w)=0. \]
If $zw'-z'w \ne 0$, we have $\rho=\lambda$ and $\lambda=\mu^{-1}$, that is $\mu \in \F_{q^k}$. This immediately implies that 
\[\mu \in S_{(u,v)}\cap S_{(u',v')}=\fq, \]
a contradiction.
Hence, $zw'=z'w$, then there exists $\xi \in \F_{q^k}$ such that $(w,z)=\xi (w',z')$, then by \eqref{eq:u+vgammau'+gammav'} it follows that 
\[ u+\gamma v=\xi (u'+\gamma v'), \]
and so $\xi \in S_{(u',v')}\cap S_{(w',z')}=\fq$ since  $(u',v')$ and $(w',z')$ are not $\fq$-proportional by \eqref{eq:muinSuvu'v'}. Hence, $(u,v)$ and $(u',v')$ are $\fq$-linearly dependent, again a contradiction. \\
\textbf{Case 3.2:} either $z=0$ or $z'=0$.\\
Without loss of generality, let $z=0$.
From the second equation of System \eqref{eq:systconduu'vv'}, we get $z'=0$.
Note that $(w,0),(w',0) \in U$. From \eqref{eq:muinSuvu'v'}  it follows that
\[ \mu (u+\gamma v)=w\,\,\,\text{and}\,\,\,\mu (u'+\gamma v')=w', \]
from which we get that
\[ (u',v')=\frac{w}{w'} (u,v). \]
Moreover, $w/w' \in S_{(w',0)}$ and so
\[ \frac{w}{w'}\in S_{(u,v)}\cap S_{(w',0)}=\fq, \]
since $v\ne 0$.
We have again a contradiction since $(u,v)$ and $(u',v')$ turn out to be $\fq$-proportional.
\end{proof}

\begin{remark}
The above theorem cannot be extended to the case of extension fields of degree two, see for instance Example \ref{ex:q=3noSidon}.
\end{remark}

\begin{remark}\label{rem:weightpoint}
By definition, it follows that
\[ \dim_{\fq}(S_{(u,v)})=\dim_{\fq}(U \cap \langle (u,v) \rangle_{\fqk}). \]
\end{remark}

The above result allows us to give the following definition.

\begin{definition}
Let $U$ be an $\fq$-subspace of $\fqk^2$. We say that $U$ has the \textbf{Sidon space property} if
for every $\fq$-linearly independent $(u,v),(u',v')\in U$ we have that 
\[S_{(u,v)}\cap S_{(u',v')}=\fq.\]
\end{definition}

As a corollary of Theorem \ref{thm:charwithNOgamma} we have the following.

\begin{corollary}\label{cor:sidonspacegamma>2}
Let $U$ be an $\fq$-subspace of $\fqk^2$. If $U$ has the Sidon space property then it defines a Sidon space in every extension of $\F_{q^k}$ of degree greater than two, that is 
\[ V_{U,\gamma}=\lbrace u+v\gamma : (u,v)\in U\rbrace \]
is a Sidon space in $\F_{q^k}(\gamma)$ for every $\gamma$ such that $[\F_{q^k}(\gamma):\mathbb{F}_{q^k}]>2$.
\end{corollary}

When $V_{U,\gamma}$ is a Sidon space, then also the related subspaces $S_{(u,v)}^\gamma$ and $S_{(u,v)}$ are Sidon spaces.

\begin{proposition}\label{prop:SuvSidon}
Let $\gamma\in\fqn$ be a root of an irreducible polynomial of degree $\frac{n}{k}\geqslant 2$ over $\fqk$. Let $U$ be an $\fq$-subspace of $\fqk^2$ of dimension $m$ over $\fq$ 
and suppose that $V=V_{U,\gamma}\in\mathcal{G}_q(n,m)$ is a Sidon space. Then
\begin{itemize}
    \item $S_{(u,v)}^\gamma$ is a Sidon space in $\fqn$, for any nonzero vector $(u,v)\in U$;
    \item if $[\F_{q^k}(\gamma):\mathbb{F}_{q^k}]>2$, then $S_{(u,v)}$ is a Sidon space in $\fqk$, for any nonzero vector $(u,v)\in U$.
\end{itemize}
\end{proposition}
\begin{proof}
Let's fix a nonzero vector $(u,v)\in U$.
Because of its definition, we have that 
\[ \alpha (u+v\gamma)\in V_{U,\gamma}, \]
for any $\alpha \in S_{(u,v)}^\gamma$, that is
\[ (u+v\gamma )S_{(u,v)}^\gamma \subseteq V_{U,\gamma}. \]
Since $V_{U,\gamma}$ is a Sidon space and every of its subspaces is a Sidon space as well, we have that $(u+v\gamma )S_{(u,v)}^\gamma$ is a Sidon space and hence $S_{(u,v)}^\gamma$ is a Sidon space.
Similar arguments can be performed for $S_{(u,v)}$.
\end{proof}

When the dimension of $V=V_{U,\gamma}$ is $k$, then it can be represented by means of a linearized polynomial.

\begin{lemma}\label{lem:expaspol}
Let $\gamma\in\fqn$ be a root of an irreducible polynomial of degree $\frac{n}{k}\geqslant 2$ over $\fqk$.
Let $U\subseteq\fqk^2$ be an $\fq$-subspace such that $\dim_{\fq}(U)=k$ and $U\cap\langle (0,1)\rangle_{\F_{q^k}}=\lbrace (0,0)\rbrace$. Let consider $V_{U,\gamma}=\lbrace a+b\gamma: (a,b)\in U\rbrace\in\mathcal{G}_q(n,k)$. Then there exists $f\in\mathcal{L}_{k,q}$ such that 
\[V_{U,\gamma}=\lbrace u+f(u)\gamma:u\in\fqk\rbrace.\]
\end{lemma}
\begin{proof}
Since $\dim_{\fq}(U)=k$ and $U\cap\langle (0,1)\rangle_{\F_{q^k}}=\lbrace (0,0)\rbrace$, for any $a \in \F_{q^k}$ there exists a unique $b \in \F_{q^k}$ such that $(a,b)\in U$.
This implies that the map 
    \[
\begin{aligned}
f\colon \fqk&\to \fqk\\
a&\mapsto b: (a,b)\in U
\end{aligned}
    \]  
is well defined and, since $V_{U,\gamma}$ is an $\fq$-subspace, it results that $f$ is an $\fq$-linear map, that is $f \in \mathcal{L}_{k,q}$. 
\end{proof}

\begin{notation}
Let $f\in\mathcal{L}_{k,q}$ be a linearized polynomial over $\fqk$ and let $U_{f}=\lbrace (u,f(u)):u\in\fqk\rbrace \subseteq \fqk^2$ be the $\fq$-subspace of $\fqk^2$ defined by $f$ which has dimension $k$ over $\fq$. Then, let denote by
\[V_{f,\gamma}=\lbrace u+f(u)\gamma:u\in\fqk\rbrace\in\mathcal{G}_q(n,k),\]
\[S_{(u,f(u))}^\gamma =S_{u,f}^\gamma\,\,\, \text{for every} \,\,\,u\in\fqk,\]
and 
\[S_{(u,f(u))} =S_{u,f}\,\,\, \text{for every} \,\,\,u\in\fqk.\]
\end{notation}

Therefore, we give the following definition on a linearized polynomial $f$, which will correspond to say that $V_{f,\gamma}$ is a Sidon space.

\begin{definition}\label{def:sidonpol}
Let $f\in\mathcal{L}_{k,q}$ be a linearized polynomial over $\fqk$. We say that $f$ is a \textbf{Sidon space polynomial with respect to $\gamma$} if for every $x_0,x_1\in\fqk$ such that $x_0,x_1$ are not $\fq$-proportional we have that $S_{x_0,f}^\gamma\cap S_{x_1,f}^\gamma=\fq$. 
Moreover, we say that $f$ is a \textbf{Sidon space polynomial} if for every $x_0,x_1\in\fqk$ such that $x_0,x_1$ are not $\fq$-proportional we have that $S_{x_0,f}\cap S_{x_1,f}=\fq$.
\end{definition}

Theorem \ref{caratterizzazioneSidonCoppie} and Corollary \ref{cor:sidonspacegamma>2} allow us to characterize Sidon spaces of form $V_{f,\gamma}$.

\begin{corollary}
\label{caratterizzazioneSidonPolinomi}
Let $f\in\mathcal{L}_{k,q}$ be a linearized polynomial over $\fqk$. Then 
\begin{enumerate}
    \item $V_{f,\gamma}$ is a Sidon space if and only if $f$ is a Sidon space polynomial with respect to $\gamma$, where $\gamma$ is a root of an irreducible polynomial of degree at least two over $\fqk$;
    \item Also, $V_{f,\gamma}$ is a Sidon space for every $\gamma$ such that $[\F_{q^k}(\gamma)\colon \F_{q^k}]>2$ if and only $f$ is a Sidon space polynomial.
\end{enumerate}
\end{corollary}

\begin{remark} \label{rk:kernelsidon}
It is interesting to observe that, because of Proposition \ref{prop:SuvSidon}, $f$ is a Sidon space polynomial then $\ker(f)$ is a Sidon space of $\fqk$.
\end{remark}

Let observe that if $f\in\mathcal{L}_{k,q}$ has the property that $\dim_{\fq}(S_{x,f}^\gamma)\leqslant 2$ for every $x\in\fqk$ then the Sidon space property is easier to check.

\begin{corollary}
\label{CorollarioCaratterizzazioneDim2}
Let $\gamma\in\fqn$ be a root of an irreducible polynomial of degree $\frac{n}{k}\geqslant 2$ over $\fqk$ and let $f\in\mathcal{L}_{k,q}$ be such that $\dim_{\fq} (S_{x,f}^\gamma)\leqslant 2$ for every $x\in\fqk$. Then $V_{f,\gamma}$ is a Sidon space if and only if $S_{x_0,f}^\gamma\neq S_{x_1,f}^\gamma$ for every $x_0,x_1\in\fqk$ such that $x_0$ and $x_1$ are not $\fq$-proportional and $\dim_{\fq}(S_{x_0,f}^\gamma)=\dim_{\fq}(S_{x_1,f}^\gamma)=2$.
\end{corollary}
\begin{proof}
Suppose that $x_0$ and $x_1$ are not $\fq$-proportional and $\dim_{\fq}( S_{x_0,f}^\gamma)=\dim_{\fq}( S_{x_1,f}^\gamma)=2$, then $1\leqslant\dim_{\fq} (S_{x_0,f}^\gamma\cap S_{x_1,f}^\gamma)\leqslant 2$. If $\dim_{\fq}(S_{x_0,f}^\gamma\cap S_{x_1,f}^\gamma)=2$ then $S_{x_0,f}^\gamma\cap S_{x_1,f}^\gamma=S_{x_0,f}^\gamma=S_{x_1,f}^\gamma$ and this is not possible because of the assumptions. Therefore $\dim_{\fq} (S_{x_0,f}^\gamma\cap S_{x_1,f}^\gamma)=1$ and by Corollary \ref{caratterizzazioneSidonPolinomi} we have that $V_{f,\gamma}$ is a Sidon space.
\end{proof}

\begin{remark}\label{rk:corwithnogamma}
From the above corollary, we also have that when $n/k>2$ we can replace $S_{x,f}^\gamma$ by $S_{x,f}$ and we can remove the dependence on $\gamma$.
\end{remark}

\begin{remark}\label{rem:Sx0,fdim}
When using linearized polynomials, we can explicitly find $S_{x_0,f}$ as the kernel of a certain linearized polynomial related to $f$. Indeed, we have
\[ S_{x_0,f}=\{ \lambda \in \fqk \colon \lambda (x_0,f(x_0))=(y_0,f(y_0)),\,\,\text{for some } y_0 \in \fqk \} \]
\[ = \{ \lambda \in \fqk \colon f(\lambda x_0)-\lambda f(x_0)=0 \}.  \]
Therefore,
\[ \dim_{\fq}(S_{x_0,f})=\dim_{\fq}(\ker(f(\lambda x_0)-\lambda f(x_0))), \]
where $f(\lambda x_0)-\lambda f(x_0)$ is seen as a linearized polynomial in $\lambda$.
Moreover, by Remark \ref{rem:weightpoint} we can compute this dimension also as follows
\[ \dim_{\fq}(S_{x_0,f})=\dim_{\fq}(\ker(f(x)-(f(x_0)/x_0) x)). \]
\end{remark}

The polynomial version of our characterization allows us to provide new constructions of Sidon spaces that we will explore in the next section.

\section{Constructions of Sidon spaces}\label{sec:constructions}

We start this section by proving that if  $f\in\mathcal{L}_{k,q}$ is scattered, then the Property 2 of Corollary \ref{caratterizzazioneSidonPolinomi} holds and $f$ is a Sidon space polynomial.

\begin{theorem}
\label{scattered implica caratterizzazione}
Let $f\in\mathcal{L}_{k,q}$ be a scattered polynomial. Then $f$ is a Sidon space polynomial.
\end{theorem}
\begin{proof}
By \eqref{eq:scattcondSa,f}, we know that for any $x_0 \in \fqk^*$ 
\[\dim_{\fq}(S_{x_0,f})=1.\]
Hence, $f$ trivially satisfies the Sidon space property.
\end{proof}

As a consequence of Corollary \ref{caratterizzazioneSidonPolinomi} we have examples of Sidon spaces.

\begin{corollary}
\label{fscatteredimplicaSidon}
Let $\gamma\in\fqn$ be a root of an irreducible polynomial of degree $\frac{n}{k}> 2$ over $\fqk$ and $f\in\mathcal{L}_{k,q}$ scattered, then 
\[ V_{f,\gamma}=\lbrace u+f(u)\gamma:u\in\fqk\rbrace \] 
is a Sidon space in $\fqn$.
\end{corollary}

\begin{remark}
In the case of $n=2k$ we cannot get the same conclusion. This seems not a limit of this argument, indeed $f(x)=x^q$ is a scattered polynomial but Example \ref{ex:q=3noSidon} provides examples of $2$-dimensional extension of $\fqk$ in which this polynomial does not define a Sidon space.
\end{remark}

We now show a characterization for monomials defining a Sidon space polynomial when $n=2k$, which clearly involves also $\gamma$.
To this aim, we recall that if $U$ and $V$ are two $\fq$-subspaces of $\fqn$, then
\[ UV=\langle uv \colon u\in U, v \in V\rangle_{\fq}\,\,\,\text{and}\,\,\, U^2=UU.  \]
Then we recall also the following theorem by Bachoc, Serra and Z\'emor in \cite{bachoc2017analogue} which gives a lower bound on the dimension of $V^2$ of a Sidon space $V$.
\begin{theorem}\cite[Theorem 18]{bachoc2017analogue} \label{lowerboundSidon}
Let $V\in\mathcal{G}_q(n,k)$ be a Sidon space of dimension $k\geqslant 3$, then
\[ 
\dim_{\fq}(V^2)\geqslant 2k.
\]
\end{theorem}
Thus we are able to prove the following characterization result. Note that the \textbf{norm} over $\fq$ of an element $a \in \mathbb{F}_{q^k}$ is $\mathrm{N}_{q^k/q}(a)=a^{\frac{q^k-1}{q-1}}$.

\begin{theorem}
Let $\gamma\in\mathbb{F}_{q^{2k}}\setminus\fqk$ and let $s\in\mathbb{N}$ such that $\gcd(s,k)=1$. Then $V_{x^{q^s},\gamma}$ is a Sidon space if and only if $\mathrm{N}_{q^{2k}/q}(\gamma)\neq 1$. 
\end{theorem}
\begin{proof}
Denote by $x^2+bx+c\in\fqk[x]$ the minimal polynomial of $\gamma$ over $\fqk$. Note that $c=\gamma^{q^k+1}=\gamma^{\frac{q^{2k}-1}{q^k-1}}$ and so 
\[ \mathrm{N}_{q^{k}/q}(c)=\mathrm{N}_{q^{2k}/q}(\gamma). \]
Therefore 
\[  \mathrm{N}_{q^{k}/q}(c)=1 \,\, if\,\, and \,\, only\,\, if\,\, \mathrm{N}_{q^{2k}/q}(\gamma)=1. \]
If $\mathrm{N}_{q^{2k}/q}(\gamma)\neq 1$, then $\mathrm{N}_{q^{k}/q}(c)\neq 1$ and  by \cite[Theorem 16]{roth2017construction}, it follows that $V_{x^{q^s},\gamma}$ is a Sidon space.\\
Conversely, suppose that $\mathrm{N}_{q^{2k}/q}(\gamma)=\mathrm{N}_{q^{k}/q}(c)=1$. Then 
\begin{equation}
\label{spanproduct}
\begin{aligned}
  V_{x^{q^s},\gamma}^2&=\langle(u+u^{q^s}\gamma)(v+v^{q^s}\gamma)\colon u,v\in\fqk \rangle_{\fq}\\
  &=\langle (uv-(uv)^{q^s}c)+(uv^{q^s}+u^{q^s}v-b(uv)^{q^s})\gamma\colon u,v\in\fqk \rangle_{\fq}. 
\end{aligned}
\end{equation}
Therefore, by \eqref{spanproduct} it follows that
\[
V_{x^{q^s},\gamma}^2\subseteq \mathrm{Im}(T)\oplus\gamma\fqk
\]
where $T:y\in\fqk \mapsto y-y^{q^s}c\in\fqk$.\\
Since $\mathrm{N}_{q^{k}/q}(c)=1$ then $\dim_{\fq}(\mathrm{Im}(T))=k-1$ and so $V_{x^{q^s},f}^2$ is contained in an $\fq$-subspace of dimension $2k-1$ of $\mathbb{F}_{q^{2k}}$. Therefore, by Theorem \ref{lowerboundSidon} it follows that $V_{x^{q^s},\gamma}^2$ cannot be a Sidon space.
\end{proof}

We resume in Table \ref{scattpoly} the list of the known examples of scattered polynomials belonging to $\mathcal{L}_{k,q}$, different entries of the tables correspond to $\Gamma \mathrm{L}(2,q^k)$-inequivalent subspaces $U_f=\{(x,f(x)) \,:\, x\in \fqk\}$. 

\begin{table}[htp]
\tabcolsep=0.2 mm
\begin{tabular}{|c|c|c|c|c|c|}
\hline
\hspace{0.5cm} & \hspace{0.2cm}$k$\hspace{0.2cm} & $f(x)$ & \mbox{Conditions} & \mbox{References} \\ \hline
i) & & $x^{q^s}$ & $\gcd(s,k)=1$ & \cite{blokhuis2000scattered} \\ \hline
ii) & & $x^{q^s}+\delta x^{q^{s(k-1)}}$ & $\begin{array}{cc} \gcd(s,k)=1,\\ \mathrm{N}_{q^k/q}(\delta)\neq 1 \end{array}$ & \cite{lunardon2001blocking,lavrauw2015solution}\\ \hline
iii) & $2\ell$ & $\begin{array}{cc}x^{q^s}+x^{q^{s(\ell-1)}}+\\ \delta^{q^\ell+1}x^{q^{s(\ell+1)}}+\delta^{1-q^{2\ell-1}}x^{q^{s(2\ell-1)}}\end{array}$ & $\begin{array}{cc} q \hspace{0.1cm} \text{odd}, \\ \mathrm{N}_{q^{2\ell}/q^\ell}(\delta)=-1,\\ \gcd(s,\ell)=1 \end{array}$ & $\begin{array}{cc}\text{\cite{BZZ,longobardi2021large}}\\\text{\cite{longobardizanellascatt,neri2022extending,ZZ}}\end{array}$\\ \hline
iv) & $6$ & $x^q+\delta x^{q^{4}}$  &  $\begin{array}{cc} q>4, \\ \text{certain choices of} \, \delta \end{array}$ & \cite{csajbok2018anewfamily,bartoli2021conjecture,polverino2020number} \\ \hline
v) & $6$ & $x^{q}+x^{q^3}+\delta x^{q^5}$ & $\begin{array}{cccc}q \hspace{0.1cm} \text{odd}, \\ \delta^2+\delta =1 \end{array}$
 & \cite{csajbok2018linearset,MMZ} \\ \hline
vi) & $8$ & $x^{q}+\delta x^{q^5}$ & $\begin{array}{cc} q\,\text{odd},\\ \delta^2=-1\end{array}$ & \cite{csajbok2018anewfamily} \\ \hline
\end{tabular}
\caption{Known examples of scattered polynomials $f$}
\label{scattpoly}
\end{table}

\begin{remark}
From Table \ref{scattpoly} we get (up to equivalence, as we will see later) all the known examples of Sidon spaces contained in the sum of two multiplicative cosets of $\fqk$; for instance \cite[Construction 11]{roth2017construction} belongs to Family i) of Table \ref{scattpoly}.
\end{remark}

\begin{remark}
In the monomial case, being scattered corresponds to being a Sidon space polynomial. Indeed, $f(x)=ax^{q^s} \in \mathcal{L}_{k,q}$ with $\gcd(s,k)=1$ is scattered and hence a Sidon space polynomial.
Suppose that $\gcd(s,k)>1$ and write $s=s' t$, where $t=\gcd(k,s)>1$.
Then $V_{f,\gamma}$ is an $\F_{q^t}$-subspace of $\F_{q^k}(\gamma)$ and hence it cannot be a Sidon space because of Theorem \ref{lem:charSidon}.
\end{remark}

We can give sufficient conditions in order to ensure that the binomial $f(x)=x^{q^i}+\delta x^{q^j}$ is a Sidon space polynomial.

\begin{proposition}
\label{Sidon caso i-j,k}
Let $f(x)=x^{q^i}+\delta x^{q^j}\in\mathcal{L}_{k,q}$, with $\gcd(k,j-i)=1$ and $j>i\geq 1$. Then $f$ is a Sidon space polynomial.
\end{proposition}
\begin{proof}
Suppose by contradiction that $f$ is not a Sidon space polynomial, i.e. there exist $x_0,x_1\in\fqk$ such that $x_0,x_1$ are not $\fq$-proportional and $\dim_{\fq}(S_{x_0,f}\cap S_{x_1,f})\geq 2$. Let $\lmb\in (S_{x_0,f}\cap S_{x_1,f})\setminus\fq$, then
\[ f(\lambda x_0)=\lambda f(x_0)\,\,\,\text{and}\,\,\,f(\lambda x_1)=\lambda f(x_1), \]
that is
\[(\lmb^{q^i}-\lmb)x_0^{q^i}=\delta x_0^{q^j}(\lmb-\lmb^{q^j})\,\,\,\text{and}\,\,\,(\lmb^{q^i}-\lmb)x_1^{q^i}=\delta x_1^{q^j}(\lmb-\lmb^{q^j}).\]
If $\lmb-\lmb^{q^i}=0$ then $\lmb^{q^j}-\lmb=0$, i.e. $\lmb^{q^{j-i}}=\lmb$ and so $\lmb\in\fqk\cap\mathbb{F}_{q^{j-i}}=\fq$ since $\gcd(k,j-i)=1$ and this is not possible. Therefore $\lmb-\lmb^{q^i}\neq 0$ and so we get
\[x_0^{q^i-q^j}=\delta\frac{\lmb-\lmb^{q^j}}{\lmb^{q^i}-\lmb}
\,\,\,\text{and}\,\,\,
x_1^{q^i-q^j}=\delta\frac{\lmb-\lmb^{q^j}}{\lmb^{q^i}-\lmb}
.\]
Thus we have
\[x_0^{q^i-q^j}=x_1^{q^i-q^j},\]
i.e. $\frac{x_0}{x_1}\in\fqk\cap\mathbb{F}_{q^{j-i}}=\fq$, that is a contradiction.
\end{proof}

Proposition \ref{Sidon caso i-j,k} gives a sufficient condition to determine if a linearized binomial defines a Sidon space polynomial, however this is not also necessary since the examples iv) and vi) in Table \ref{scattpoly} have the property that the difference of the exponents of the $q$-th powers that appear divides $k$ and since they are scattered, by Theorem \ref{scattered implica caratterizzazione}, they are also Sidon space polynomials.
We will now show that in the case in which $(i-j,k)\ne 1$, a binomial as in Proposition \ref{Sidon caso i-j,k} gives a Sidon space polynomial if and only if the binomial is also scattered.
\begin{theorem}\label{th:binsidoniffscattered}
Let $f(x)=x^{q^i}+\delta x^{q^j}\in\mathcal{L}_{k,q}$, with $\gcd(k,j-i)>1$, $\delta \ne 0$ and $j>i\geq 1$. 
Let $j=i+st$, where $s,t \in \mathbb{N}$ are such that $\gcd(s,k)=1$, $t>1$ and $t \mid k$.
Then
\begin{itemize}
    \item[i)] $f$ is a Sidon space polynomial if and only if $f$ is scattered;
    \item[ii)] if $\gcd(i,t)\ne 1$ then $f$ is not a Sidon space polynomial;
    \item[iii)] if $\N_{q^k/q^t}(\delta)=(-1)^{k/t}$ then $f$ is not a Sidon space polynomial. 
\end{itemize}
\end{theorem}
\begin{proof}
i) We need to prove that if $f$ is a Sidon space polynomial then $f$ is scattered.
By contradiction, assume that $f$ is not scattered, hence there exists $x_0 \in \fqk^*$ such that
\[ \dim_{\fq}(S_{x_0,f})>1. \]
Consider $\rho \in \F_{q^t} \setminus \fq$ and let $x_1=\rho x_0$, then it is easy to check that 
\[ S_{x_0,f}=\{ \lambda \in \fqk \colon -\lambda (x_0^{q^i}+\delta x_0^{i+st})+\lambda^{q^i}x_0^{q^i}+\delta \lambda^{q^{i+st}}x_0^{q^{i+st}}=0 \} \]
and
\[ S_{x_1,f}=\{ \lambda \in \fqk \colon -\lambda (x_1^{q^i}+\delta x_1^{i+st})+\lambda^{q^i}x_1^{q^i}+\delta \lambda^{q^{i+st}}x_1^{q^{i+st}}=0 \}\]
coincide and hence $S_{x_0,f}\cap S_{x_1,f}=S_{x_0,f} \supset \fq$, a contradiction to the fact that $f$ is a Sidon space polynomial.\\
ii) If $\ell=\gcd(i,t)\ne 1$ then $\ell \mid j$ and so $V_{f,\gamma}$ is an $\F_{q^\ell}$-subspace of $\fqn$ and hence by Theorem \ref{lem:charSidon} it is not a Sidon space.\\
iii) If $\N_{q^k/q^t}(\delta)=(-1)^{k/t}$, then the equation $x^{q^i(1-q^{ts})}=-\delta$ admits a solution $x_0$ in $\F_{q^k}$ and the set of its solutions is $\{ x_0 \epsilon \colon \epsilon \in \F_{q^t}^* \}$. This implies that  $\ker(f)=\{x \in \F_{q^k} \colon x^{q^i}=-\delta x^{q^{i+ts}}\}=x_0 \F_{q^t}$ and hence $f$ is not a Sidon space polynomial by Remark \ref{rk:kernelsidon}.
\end{proof}
\begin{remark}
In the previous result, we characterize the property of a binomial of being a Sidon space polynomial in terms of scattered polynomials, when $i-j$ is not coprime with $k$. Therefore, classification results on binomials that are scattered can be given now for Sidon space polynomials. 
A special family considered in the scattered polynomials framework is
\[ f(x)=x^{q^i}+\delta x^{q^{i+t}}, \]
where $\gcd(i,k)=1$ and $k=2t$, originally introduced in \cite{csajbok2018anewfamily}.
In \cite{bartoli2021conjecture} and \cite{polverino2020number}, there is an explicit condition on $\delta$ in order to characterize the scattered $f$ when $t=3$. When $t=4$ in \cite[Theorem 1.1]{timpanella2023family}, $f$ is scattered if and only if $\delta^{q^4+1}=-1$ when $q$ is a large odd prime power. Moreover, in \cite[Theorem 1.1]{polverino2021certain} it has been proved that if $k$ is larger than $4i+2$ then $f$ is not scattered.
\end{remark}

Thanks to the Proposition \ref{Sidon caso i-j,k} we are able to introduce new examples of Sidon spaces via some linearized polynomials which are not scattered.

\begin{corollary}
Let $f(x)=x^{q^s}+\delta x^{q^{2s}} \in \mathcal{L}_{k,q}$ with $\gcd(s,k)=1$, then $f$ is a Sidon space polynomial.
\end{corollary}

\begin{remark}
When $s=1$ and $k\geq 5$, $f(x)=x^{q}+\delta x^{q^{2}}$ with $\delta \ne 0$ is not scattered; see \cite[Remark 3.11]{zanella2019condition} and also \cite{montanucci2022class}.
\end{remark}

\begin{remark}
The polynomials considered in Proposition \ref{Sidon caso i-j,k} can be very far from being scattered, that is it may exist $x_0$ for which the dimension of $S_{x_0,f}$ is large. To this aim recall from \cite[Theorem 1.1]{mcguiretrinomials} (see also \cite[Theorem 1.3]{santonastaso2022linearized}) that if $k=(d-1)d+1$ and $q=p^h$ (for some prime $p$ and $h \in \mathbb{N}$) then $\dim_{\fq}(\ker(ax+bx^q-x^{q^d}))=d$ if and only if 
\[
\left\{
\begin{array}{lll}
\mathrm{N}_{q^k/q}(a)=(-1)^{d-1},\\
b=-a^{qf_1}\,\,\,\text{where}\,\,\, f_1=\sum_{i=0}^{d-1} q^{id},\\
d-1\,\, \text{is a power of }p.
\end{array}
\right.
\]
Therefore, choosing $d=p^\ell+1$, $q=p^h$, $k=(d-1)d+1$, $j=d$ and $i=1$ by the above result we have the existence of a linearized polynomial of the form $mx+ x^q+\delta x^{q^{d}}$ such that 
\[ \dim_{\fq}(\ker(mx+ x^q+\delta x^{q^{d}}))=d, \]
and by Remark \ref{rem:Sx0,fdim} this means that there exists $x_0 \in \fqk^*$ such that 
\[ \dim_{\fq}(S_{f,x_0})=d, \]
where $f(x)=x^q+\delta x^{q^{d}}$ and satisfies the assumptions of Proposition \ref{Sidon caso i-j,k}.
\end{remark}

Let now see another example which regards the Lunardon-Polverino polynomials which, under certain assumptions, are Sidon space polynomials also in the case in which they are not scattered.

\begin{corollary}
Let $f(x)=x^{q^s}+\delta x^{q^{s(k-1)}}\in \mathcal{L}_{k,q}$ with $\gcd(s,k)=1$. Then $f$ is a Sidon space polynomial if and only if either $\mathrm{N}_{q^k/q}(\delta)\neq 1$ or $k$ is odd.  
\end{corollary}
\begin{proof}
By \cite[Theorem 3.4]{zanella2019condition}, $f$ is a scattered polynomial if and only if $\mathrm{N}_{q^k/q}(\delta)\neq 1$.
So, if $\mathrm{N}_{q^k/q}(\delta)\neq 1$ then Theorem \ref{scattered implica caratterizzazione} implies that $f$ is a Sidon space polynomial.\\
If $k$ is odd, since $\gcd(k,s-s(k-1))=1$, Proposition \ref{Sidon caso i-j,k} implies that $f$ is a Sidon space polynomial.
If $k$ is even, then Theorem \ref{th:binsidoniffscattered} implies that $f$ is a Sidon space polynomial if and only if $f$ is scattered, and hence if and only if $\mathrm{N}_{q^k/q}(\delta)\neq 1$.
\end{proof}

\section{Direct sum of Sidon spaces}

In \cite[Lemma 3.1]{li2021cyclic}, the authors proved that, under certain circumstances, the sum of two Sidon spaces is a Sidon space as well.
So, \cite[Lemma 3.1]{li2021cyclic} can be stated as follows: if $U$ and $V$ are two  Sidon spaces of $\fqn$ satisfying the following two properties:
\begin{itemize}
    \item $(U+V)^2=U^2\oplus UV\oplus V^2$;
    \item $\dim_{\fq}(U \cap \alpha V) \leq 1$ for each $\alpha \in \F_{q^n}^*$,
\end{itemize}
then $U+V$ is a Sidon space of $\fqn$. The main problem is that it is not easy to find two Sidon spaces satisfying both these assumptions.
In this section we will weaken the assumptions and we will prove a similar result, but the direct sum will be contained in an extension of $\fqn$.



\begin{theorem}\label{thm:directsumSidon}
Let $V_1 \in \mathcal{G}_q(n,k_1)$, $V_2 \in \mathcal{G}_q(n,k_2)$ be two distinct Sidon spaces. Suppose that $V_1$ and $V_2$ satisfy 
\begin{equation}\label{eq:mutuallySidon}
\dim_{\fq}(V_1 \cap \alpha V_2) \leq 1\,\,\, \text{for each}\,\,\, \alpha \in \F_{q^n}^*.
\end{equation} 
Let $m$ be a multiple of $n$ such that $\frac{m}{n}>2$. Let $\delta$ be a root of an irreducible polynomial over $\F_{q^n}$ of degree $\frac{m}{n}$. Then 
\[V=\{v_1+\delta v_2:v_1 \in V_1,v_2 \in V_2\}\] 
is a Sidon space of $\mathcal{G}_q(m,k_1+k_2)$.
\end{theorem}
\begin{proof}
We start by noting that
\[ V=V_{V_1\times V_2,\delta}. \]
By Theorem \ref{thm:charwithNOgamma}, it is enough to prove that $V_1\times V_2 \subseteq \fqn\times \fqn$ has the Sidon space property, that is for any $(u_1,u_2),(u_1',u_2') \in V_1\times V_2$ which are not $\fq$-proportional we have
\[ S_{(u_1,u_2)}\cap S_{(u_1',u_2')}=\fq. \]
Let fix $(u_1,u_2),(u_1',u_2') \in V_1\times V_2$ which are not $\fq$-proportional and suppose that there exists $\rho \in \F_{q^n}\setminus \fq$ which is in $S_{(u_1,u_2)}\cap S_{(u_1',u_2')}$.
This implies the existence of $(w_1,w_2),(w_1',w_2') \in V_1\times V_2$ such that
\[ \rho(u_1,u_2)=(w_1,w_2)\,\,\,\text{and}\,\,\,\rho(u_1',u_2')=(w_1',w_2'), \]
so
\begin{equation}\label{eq:systcondrhouvu'v'}
\left\{
\begin{array}{llll}
\rho u_1=w_1,\\
\rho u_2=w_2,\\
\rho u_1'=w_1',\\
\rho u_2'=w_2'.
\end{array}
\right.
\end{equation}
and from last equation we get that either $u_2'=w_2=0$ or $u_2',w_2'$ are both non zero.\\ 
\textbf{Case 1:} Suppose that $u_2'$ and $w_2'$ are both nonzero. Then from System \eqref{eq:systcondrhouvu'v'} we get
\begin{equation}\label{eq:systconddivisioneu'v'}
\left\{
\begin{array}{lll}
\frac{u_1}{u_2'}=\frac{w_1}{w_2'}=\lmb,\\
 \frac{u_2}{u_2'}=\frac{w_2}{w_2'}=\mu,\\
 \frac{u_1'}{u_2'}=\frac{w_1'}{w_2'}=\xi.
\end{array}
\right.
\end{equation}
which implies that
\begin{equation}\label{eq:conddimensionisidon}
\begin{aligned}
&u_1,w_1\in V_1\cap \lmb V_2,\\
&u_2,w_2\in V_2\cap\mu V_2\\
&u_1',w_1'\in V_1\cap\xi V_2.  
\end{aligned}
\end{equation}
If $(u_1',w_1')=(u_1,w_1)=(0,0)$ then $(u_2,w_2)\neq(0,0)$. Moreover, if $\mu\in\fq$ then from System \eqref{eq:systconddivisioneu'v'} it follows that $(u_1,u_2)=(0,u_2)=\mu(0,u_2')=\mu(u_1',u_2')$, i.e. $(u_1,u_2)$ and $(u_1',u_2')$ are $\fq$-proportional, a contradiction. If $\mu\notin\fq$, since $V_2$ is a Sidon space, from Equation \eqref{eq:conddimensionisidon} it follows that $u_2$ and $w_2$ are $\fq$-proportional, i.e. $\rho\in\fq$, a contradiction.\\
Therefore $(u_1',w_1')\neq(0,0)$ or $(u_1,w_1)\neq (0,0)$, from System \eqref{eq:systconddivisioneu'v'} this implies that $\xi\neq 0$ or $\lmb\neq 0$. \\
If $\xi\neq 0$ then by Condition \eqref{eq:mutuallySidon} on $V_1$ and $V_2$ we have that $u_1'$ and $w_1'$ are $\fq$-proportional, i.e. $\rho\in\fq$, a contradiction. \\
If $\lmb\neq 0$ then by Condition \eqref{eq:mutuallySidon} on $V_1$ and $V_2$ we have that $u_1$ and $w_1$ are $\fq$-proportional, i.e. $\rho\in\fq$, a contradiction.\\
\textbf{Case 2:} Suppose that $(u_2',w_2')=(0,0)$, then $u_1',w_1'\neq 0$, then from System \eqref{eq:systcondrhouvu'v'} (as done for \eqref{eq:systconddivisioneu'v'}) we get
\begin{equation}\label{eq:syst2conddivisioneu'v'}
\left\{
\begin{array}{ll}
\frac{u_1}{u_1'}=\frac{w_1}{w_1'}=\sigma,\\
 \frac{u_2}{u_1'}=\frac{w_2}{w_1'}=l.
\end{array}
\right.
\end{equation}
from which we derive
\begin{equation}\label{eq:conddimensionisidon2}
\begin{aligned}
&u_1,w_1\in V_1\cap \sigma V_1,\\
&u_2,w_2\in V_2\cap l V_1.
\end{aligned}
\end{equation}
If $l=0$ then $(u_2,w_2)=(0,0)$ and $u_1,w_1 \ne 0$ and from Equation \eqref{eq:conddimensionisidon2}, if $\sigma\in\fq$ then $(u_1,u_2)=(u_1,0)$ and $(u_1',u_2')=(u_1',0)$ are $\fq$-proportional, a contradiction. If $\sigma\notin\fq$, since $V_1$ is a Sidon space, we have $u_1$ and $w_1$ are $\fq$-proportional, implying that $\rho \in \fq$.\\
If $l\neq 0$, we can argue as before and get a contradiction also in this case.
\end{proof}

\begin{remark}
The above result can also be proved by using \cite[Lemma 3.1]{li2021cyclic} by taking as the two subspaces $V_1$ and $\delta V_2$, however this version comes quite naturally from our previous discussions.
\end{remark}

\begin{remark}
Let observe that if $V_1=V_2=V$ then $V+\delta V$ is not a Sidon space. Indeed, let $\alpha\in\fqn\setminus \fq$ such that $T=V\cap \alpha V$ has dimension $1$ (which always exists if we require that $U$ is a strictly $\fq$-subspace). Consider now $T+\delta T$, which has dimension $2$, and is contained in $(V+\delta V)\cap \alpha (V+\delta V)$ and hence $V+\delta V$ cannot be a Sidon space.
\end{remark}

As a consequence of Theorem \ref{thm:directsumSidon}, we can show some examples of Sidon spaces, which are contained in the sum of two multiplicative cosets of $\fqn$ (so belonging to Family 1 of the Introduction), but different from the one we constructed by using linearized polynomials in the previous section.

In order to use Theorem \ref{thm:directsumSidon}, we need to find Sidon spaces which satisfy Condition \eqref{eq:mutuallySidon}. For instance, we can consider the subspaces of \cite[Construction 37]{roth2017construction}.
They prove that the subspaces in \cite[Construction 37]{roth2017construction} satisfy the conditions of Theorem \ref{thm:directsumSidon} when $n=2k$, but it is easy to see that the proof still works when replacing $n/k>2$ and $\gamma_0$ with any element in $\fqn\setminus\fqk$.

\begin{corollary}\label{cor:doublepseudo}
Let $q\geq 3$, $n,k$ be two positive integers with $k\mid n$ and $n/k\geq 2$, let $w$ be a primitive element of $\F_{q^k}$ and let $\gamma \in \fqn\setminus \fqk$ (if $n=2k$ then see \cite[Construction 37]{roth2017construction} for the assumptions on $\gamma$).
Let $m$ be a multiple of $n$ such that $\frac{m}{n}>2$ and let $\delta$ be a root of an irreducible polynomial over $\F_{q^n}$ of degree $\frac{m}{n}$. 
Then 
\[V=\{(x+x^q\gamma)+\delta (y+y^q\gamma w): x,y \in \fqk\}\] 
is a Sidon space in $\F_{q^m}$.
\end{corollary}

In particular, if $m$ is odd, the examples of the above corollary is contained in $\F_{q^n}+\delta \F_{q^n}$ and has dimension $2k$, but $\F_{q^{2k}}$ is not a subfield of $\F_{q^m}$ and hence it is different from the construction exhibited in the previous section.

\section{Equivalence of Sidon spaces}

As already explained in the Introduction, since Sidon spaces and cyclic subspace codes with a certain minimum distance are equivalent objects, it is quite natural to give a definition of equivalence for Sidon spaces arising from the equivalence of subspace codes; see \cite{zullo2023multi}. 
The study of the equivalence for subspace codes was initiated by Trautmann in \cite{trautmann2013isometry} and the case of cyclic subspace codes has been investigated in \cite{gluesing2021automorphism} by Gluesing-Luerssen and Lehmann.
Therefore, motivated by \cite[Theorem 2.4]{gluesing2021automorphism}, we say that two cyclic subspace codes $\mathrm{Orb}(U)$ and $\mathrm{Orb}(V)$ are \textbf{linearly equivalent} if there exists $i \in \{0,\ldots,n-1\}$ such that 
\[ \mathrm{Orb}(U)=\mathrm{Orb}(V^{q^i}), \]
where $V^{q^i}=\{ v^{q^i} \colon v \in V \}$.
This happens if and only if $U=\alpha V^{q^i}$, for some $\alpha \in \fqn^*$.
We can replace the action of the Frobenius maps $x\in \fqn\mapsto x^{q^i}\in \fqn$ with any automorphism in $\mathrm{Aut}(\fqn)$, since this will still preserve the metric properties of the codes.
Therefore, we say that two cyclic subspace codes $\mathrm{Orb}(U)$ and $\mathrm{Orb}(V)$ are \textbf{semilinearly equivalent} if there exists $\sigma \in \mathrm{Aut}(\fqn)$ such that 
\[ \mathrm{Orb}(U)=\mathrm{Orb}(V^{\sigma}), \]
where $V^{\sigma}=\{ v^\sigma \colon v \in V \}$ and this happens if and only if $U=\alpha V^{\sigma}$, for some $\alpha \in \fqn^*$.
Therefore, we can give the following definition of equivalence among two Sidon spaces.

\begin{definition}
Let $U$ and $V$ be two Sidon spaces in $\fqn$. Then we say that $U$ and $V$ are \textbf{semilinearly equivalent} if the associated codes $\mathrm{Orb}(U)$ and $\mathrm{Orb}(V)$ are semilinearly equivalent, that is there exist $\sigma \in \mathrm{Aut}(\fqn)$ and $\alpha \in \fqn$ such that $U=\alpha V^{\sigma}$.
In this case, we will also say that they are equivalent under the action of $(\alpha, \sigma)$.
\end{definition}

We can characterize the equivalence between two Sidon spaces which are contained in the sum of two multiplicative cosets of a subfield of $\fqn$ as follows.

\begin{theorem}\label{thm:equiv}
Let $k$ and $n$ be two positive integers such that $k \mid n$.
Let $U,W$ be two $m$-dimensional $\fq$-subspaces of $\mathbb{F}_{q^k}^2$. Consider 
\begin{center}
    $V_{U,\gamma}=\lbrace u+u'\gamma: (u,u')\in U\rbrace$\\
    $V_{W,\xi}=\lbrace w+w'\xi: (w,w')\in W\rbrace$
\end{center}
where $\gamma,\xi\in\fqn$ are such that $\lbrace 1,\gamma\rbrace$ and $\lbrace 1,\xi\rbrace$ are $\fqk$-linearly independent and suppose that they are not contained in any multiplicative coset of $\F_{q^k}$. Then $V_{U,\gamma}$ and $V_{W,\xi}$ are semilinearly equivalent under the action of $(\lmb,\sigma)\in\fqn^*\times \mathrm{Aut}(\fqn)$ if and only if there exists $A=\left(\begin{aligned}
    \begin{matrix}
        c & d \\
        a & b 
    \end{matrix}
\end{aligned}\right) \in \mathrm{GL}(2,\fqk)$ such that $\xi=\frac{a+b\gamma^\sigma}{c+d\gamma^\sigma}$, $\lmb=\frac{1}{c+d\gamma^\sigma}$ and $U^{\sigma}=\{w A \colon w \in W\}=W \cdot A$.
\end{theorem}
\begin{proof}
Suppose that $V_{U,\gamma}$ and $V_{W,\xi}$ are semilinearly equivalent, i.e. there exist $\lmb\in\fqn^*$ and $\sigma\in \mathrm{Aut}(\fqn)$ such that $\lmb V_{U,\gamma}^{\sigma}=V_{W,\xi}$. Note that
\[\lmb V_{U,\gamma}^{\sigma}=\lbrace \lmb u+\lmb u'\gamma^{\sigma}: (u,u')\in U\rbrace \subseteq \lmb\fqk+\lmb\gamma^\sigma\fqk\]
and
\[V_{W,\xi}=\lbrace w+w'\xi:(w,w')\in W\rbrace\subseteq\fqk+\xi\fqk.\]
Since $V_{W,\xi}=\lmb V_{U,\gamma}^{\sigma}$, then 
\[V_{W,\xi}=\lmb V_{U,\gamma}^{\sigma}\subseteq(\lmb\fqk+\lmb\gamma^{\sigma}\fqk)\cap(\fqk+\xi\fqk).\]
If $\dim_{\F_{q^k}}((\lmb\fqk+\lmb\gamma^{\sigma}\fqk)\cap(\fqk+\xi\fqk))=1$,  we have $\lmb V_{U,\gamma}^{\sigma}=V_{W,\xi}\subseteq (\lmb\fqk+\lmb\gamma^{\sigma}\fqk)\cap(\fqk+\xi\fqk)=\alpha \F_{q^k}$, for some $\alpha \in \fqn$, this is not possible because $V_{W,\xi}$ cannot be contained in a multiplicative coset of $\F_{q^k}$. Therefore, $\dim_{\F_{q^k}}(\lmb\fqk+\lmb\gamma^{\sigma}\fqk)\cap(\fqk+\xi\fqk)=2$ and $\lmb\fqk+\lmb\gamma^{\sigma}\fqk=\fqk+\xi\fqk$, i.e. 
\[\lmb\langle 1,\gamma^{\sigma}\rangle_{\fqk}=\langle 1,\xi\rangle_{\fqk}.\]
Then 
\[\xi=\lmb(a+b\gamma^{\sigma})\,\,\, \text{with}\,\,\, a,b\in\fqk\]
and
\[1=\lmb(c+d\gamma^{\sigma})\,\,\, \text{with}\,\,\,c,d\in\fqk.\]
Hence $\lmb=\frac{1}{c+d\gamma^{\sigma}}$, $\xi=\lmb(a+b\gamma^{\sigma})=\frac{a+b\gamma^{\sigma}}{c+d\gamma^{\sigma}}$ and $A=\left(\begin{aligned}
\begin{matrix}
    c & d \\
    a & b
\end{matrix}
\end{aligned}\right)$ is invertible since $\xi \notin \F_{q^k}$.
Moreover, since $\lmb V_{U,\gamma}^{\sigma}=V_{W,\xi}$ and $1, \gamma^{\sigma}$ are $\fqk$-linearly independent, it follows that
\begin{equation}
    \begin{aligned}
        \lmb^{-1}V_{W,\xi}&=\lbrace \lmb^{-1}w+\lmb^{-1}w'\xi: (w,w')\in W\rbrace \\
        &=\left\lbrace (c+d\gamma^{\sigma})w+(c+d\gamma^{\sigma})w'\frac{a+b\gamma^{\sigma}}{c+d\gamma^{\sigma}}: (w,w')\in W\right\rbrace \\
        &=\lbrace (c+d\gamma^{\sigma})w+(a+b\gamma^{\sigma})w':(w,w')\in W\rbrace,\\
    \end{aligned}
\end{equation}
since $\lmb V_{U,\gamma}^{\sigma}=V_{W,\xi}$, we get
\[U^{\sigma}=W \cdot A\]
hence we obtain the first part of the assertion.\\
Conversely, suppose that there exists $A=\left(\begin{aligned}
    \begin{matrix}
        c & d \\
    a & b 
    \end{matrix}
\end{aligned}\right) \in \mathrm{GL}(2,\fqk)$ such that $\xi=\frac{a+b\gamma^\sigma}{c+d\gamma^\sigma}$ and $U^{\sigma}=W\cdot A$. Let $\lambda=\frac{1}{c+d\gamma^\sigma}$.
Then it follows that
\begin{equation}
    \begin{aligned}
    V_{W,\xi}&=\lbrace w+w'\xi: (w,w')\in W\rbrace\\
    &=\left\lbrace w+w'\frac{a+b\gamma^{\sigma}}{c+d\gamma^{\sigma}}: (w,w')\in W\right\rbrace\\
    &=\lmb \lbrace (c+d\gamma^{\sigma})w+(a+b\gamma^{\sigma})w':(w,w')\in W\rbrace\\
    &=\lmb \lbrace (cw+aw')+(dw+bw')\gamma^{\sigma}:(w,w')\in W\rbrace \\
    &=\lmb \lbrace u^{\sigma}+u'^{\sigma}\gamma^{\sigma}: (u,u')\in U\rbrace\\
    &=\lmb V_{U,\gamma}^{\sigma}.
    \end{aligned}
\end{equation}
hence $V_{U,\gamma}$ and $V_{W,\xi}$ are semilinearly equivalent under the action of $(\lambda,\sigma)$.
\end{proof}

\begin{remark}
A first consequence of Theorem \ref{thm:equiv} is that all the known examples (\cite[Constructions 11 and 15]{roth2017construction}, \cite[Lemma 3.1]{zhang2023constructions}, \cite[Lemmas 4.1 and 4.2]{li2021cyclic} and \cite[Lemma 2.4]{zhang2023further}) of Sidon spaces contained in the sum of two multiplicative cosets of a field are equivalent to $V_{x^{q^s},\gamma}$.
For instance, consider the example shown in \cite[Lemma 2.4]{zhang2023further}, that is $V_{U,\gamma}\subseteq \fqn$ where
\[U=\{ (x,(x^{q^\ell}+ax)b) \colon x \in \fqk \},\]
for some $a,b \in \fqk$ with $b\ne 0$, $\gamma \in \fqn \setminus \fqk$ and $\gcd(\ell,k)=1$.
Then it is easy to see that
\[  U \begin{pmatrix} 
1 & -a \\0 & b^{-1}
\end{pmatrix} =U_{x^{q^\ell}}.
\]
Therefore, by Theorem \ref{thm:equiv} the subspace $V_{U,\gamma}$ is equivalent to $V_{x^{q^{\ell}},\gamma'}$ where $\gamma' =\frac{b^{-1}\gamma}{1-a\gamma}$. 
\end{remark}

Another important consequence is that every Sidon space of dimension $k$ in an extension of $\mathbb{F}_{q^k}$ which is contained in the sum of two multiplicative cosets of $\mathbb{F}_{q^k}$ is equivalent to a Sidon space defined by a linearized polynomial in $\mathcal{L}_{k,q}$.

\begin{corollary}
Let $k$ and $n$ be two positive integers such that $k \mid n$.
Let $U$ be a $k$-dimensional $\fq$-subspace of $\mathbb{F}_{q^k}^2$. Consider $V=V_{U,\gamma}$, where $\gamma \in\fqn \setminus \F_{q^k}$. Then $V$ is semilinearly equivalent to $V_{f,\xi}$ for some $f\in \mathcal{L}_{k,q}$ and $\xi \in \fqn\setminus \F_{q^k}$.
\end{corollary}
\begin{proof}
Note that, since $\dim_{\fq}(U)=k$ then \cite[Lemma 4.6]{polverino2023divisible} there exists at least one $(a,b)\in \F_{q^k}\times \F_{q^k}\setminus \{(0,0)\}$ such that
\[ U\cap \langle (a,b)\rangle_{\F_{q^k}}=\{(0,0)\}. \]
Consider any $\phi \in \mathrm{G L}(2,q^k)$ such that $\varphi((a,b))=(0,1)$. In this way,
\[ \phi(U)\cap \langle (0,1)\rangle_{\F_{q^k}}=\{(0,0)\}, \]
and Lemma \ref{lem:expaspol} implies the assertion.
\end{proof}

Also from Theorem \ref{thm:equiv} it follows that if $U,W \subseteq \fqk^2$ are $\fq$-subspaces of dimension $k$ satisfying the Sidon space property which are $\mathrm{\Gamma L}(2,q^k)$-inequivalent, then $V_{U,\gamma}$ and $V_{W,\gamma}$ are semilinearly inequivalent for any $\gamma \in \fqn \setminus \fqk$. Therefore, scattered polynomials in Table \ref{scattpoly} can be used to construct semilinearly inequivalent examples of Sidon spaces (via Theorem \ref{scattered implica caratterizzazione}), since different entries of the table give rise to $\mathrm{\Gamma L}(2,q^k)$-inequivalent subspaces. Moreover, every entry of Table \ref{scattpoly} may contain $\mathrm{\Gamma L}(2,q^k)$-inequivalent subspaces. 
More precisely, denote by $\Lambda (f)$ the number of $\mathrm{\Gamma L}(2,q^k)$-inequivalent subspaces to $U_f$ in the corresponding family in the Table \ref{scattpoly}. 
Family i) in Table \ref{scattpoly} is known as pseudoregulus type and
\[\Lambda(x^{q^s})=\varphi(k)/2,\]
where $\varphi$ is the Euler's totient function.
Polynomials in Family ii) in Table \ref{scattpoly} are known as Lunardon-Polverino polynomials and 
\[\Lambda(x^{q^s}+\delta x^{q^{s(k-1)}})\leq \left\{\begin{array}{cc} 
\varphi(k)\frac{q-2}2 & \text{if}\,\, 2 \nmid k,\\
& \\
\varphi(k)\frac{(q+1)(q-2)}4 & \text{if}\,\, 2 \mid k,
\end{array}\right.\]
see \cite[Section 2.2]{longobardi2021large} and for further bounds see \cite[Theorem 4.3]{tang2023automorphism}.
For the Family iii) we have
\[ \Lambda(x^{q^s}+x^{q^{s(\ell-1)}}+\delta^{q^\ell+1}+\delta^{1-q^{2\ell-1}}x^{q^{s(2\ell-1)}})=
\left\lfloor\dfrac{\varphi(\ell)(q^\ell+1)}{4r\ell(q^2+1)}\right\rfloor 
\]
if $\ell\equiv 2 \pmod{4}$, and
\[ \Lambda(x^{q^s}+x^{q^{s(\ell-1)}}+\delta^{q^\ell+1}+\delta^{1-q^{2\ell-1}}x^{q^{s(2\ell-1)}})=
\left\lfloor\dfrac{\varphi(\ell)(q^\ell+1)}{8r\ell}\right\rfloor
\]
if $\ell \not\equiv 2 \pmod{4}$, where $q=p^r$ and $p$ is prime, see \cite[Corollary 4.3]{longobardi2021large} and \cite[Theorem 4.12]{neri2022extending}.

Also, by Theorem \ref{thm:equiv}, from the same polynomial $f$ we can get inequivalent Sidon spaces $V_{f,\gamma}$ and $V_{f,\xi}$ fixing $\gamma,\xi \in \fqn \setminus \fqk$ such that $(1,\gamma),(1,\xi) \in \fqn\times\fqn$ are $\mathrm{\Gamma L}(2,q^k)$-inequivalent.
Now, we give a lower bound on the number $\Gamma(\gamma)$ of elements $\xi \in \fqn \setminus \fqk$ giving inequivalent Sidon spaces of the form $V_{f,\xi}$, where $f$ is fixed.
The elements $\xi \in \fqn \setminus \fqk$ which gives equivalent Sidon spaces to $V_{f,\gamma}$ are of the form
\[\xi = \frac{a+b\gamma^\sigma}{c+d\gamma^\sigma},\]
for some $(a,b,c,d)\in \fqk$ such that $ad-bc \ne 0$ and $\sigma \in \mathrm{Aut}(\fqn)$.
Therefore this number is greater than or equal to the size of $L$ minus one, where
\[L=\{ \langle ( a+b\gamma^\sigma,c+d\gamma^\sigma ) \rangle_{\fqn} \colon a,b,c,d \in \fqk, (a,b,c,d)\ne (0,0,0,0) \}.\]
Then $L$ is an $\fqk$-linear set of the form of those studied in \cite{jena2021linear} and hence $|L|=q^{3k}+1$ when $n/k>3$; we refer to \cite{lavrauw2015field,polverino2010linear} for surveys on linear sets.
Hence,
\[ \Gamma(\gamma) \geq q^n-q^{3k}|\mathrm{Aut}(\fqn)|. \]

For instance, the number of inequivalent Sidon spaces we can get from Family i) is 
\[ \Lambda(x^{q^s})\Gamma(\gamma)\geq \frac{\varphi(k)}{2} (q^n-q^{3k}|\mathrm{Aut}(\fqn)|). \]

\subsection{Sidon spaces as kernel of subspace polynomials}

In this section we will show that the examples of Sidon spaces we found in this paper are not equivalent to the Sidon spaces described as the kernel of a subspace trinomial, hence they are not the ones described in \cite{otal2017cyclic,santonastaso2022linearized}.

In the next proposition, we show how to read the semilinear equivalence of $\fq$-subspaces on the related subspace polynomial.
We just recall that if $F(x)=\sum_{i=0}^k a_i x^{q^i}\in \mathcal{L}_{n,q}$ and $\sigma \in \mathrm{Aut}(\fqn)$, then $F^{\sigma}(x)=\sum_{i=0}^k a_i^{\sigma} x^{q^i}$.

\begin{proposition}\label{prop:equivsubpol}
Let $V$ be a $k$-dimensional $\fq$-subspace of $\fqn$ and let $F \in \mathcal{L}_{n,q}$ be the associated subspace polynomial, that is $F$ is a monic polynomial of $q$-degree $k$ and $\ker(F)=V$. 
Every $\fq$-subspace semilinearly equivalent to $V$ is the kernel of
\[ \lambda^{q^k}F^\sigma(\lambda^{-1}x) \]
for some $\lambda \in \fqn$ and $\sigma \in \mathrm{Aut}(\fqn)$.
\end{proposition}
\begin{proof}
Any $\fq$-subspace semilinearly equivalent to $V$ is of the form $\lambda V^{\sigma}$, for some $\lambda \in \fqn^*$ and $\sigma \in \mathrm{Aut}(\fqn)$.
Then it is easy to see that $\lambda^{q^k}F^\sigma(\lambda^{-1}x)$ is monic, it has $q$-degree $k$ and $\lambda^{q^k}F^\sigma(\lambda^{-1}u)=0$, for every $u \in \lambda V^{\sigma}$.
\end{proof}

We now prove that the examples found in \cite{roth2017construction} are not semilinearly equivalent to the examples found in \cite{otal2017cyclic,santonastaso2022linearized}.

\begin{theorem}
    \label{casomonomio}
Let $\gamma\in\fqn$ be a root of an irreducible polynomial of degree $\frac{n}{k}\geqslant 2$ over $\fqk$.
Let $\delta=1-(-1)^{k-2}\gamma^{q^{k-1}+\dots+1}$,
\[ a_0=-\delta^q,\,\,\, a_1=-(\gamma^{q^k}-\gamma),\,\,\, a_k=\delta \]
and
\[a_i=(-1)^i(\gamma^{q^{i-1}+\dots+q})(\gamma^{q^k}-\gamma),\]
for $i\in\lbrace 2,\dots,k-1\rbrace$.\\
Let $F(x)=\displaystyle\sum_{i=0}^k a_ix^{q^i}\in\mathcal{L}_{n,q}$.
Then $\ker(F)=V_{x^q,\gamma}=\lbrace u+u^q\gamma:u\in\fqk\rbrace$.
\end{theorem}
\begin{proof}
Let's start by observing that since $\dim_{\fq}(\ker(F))\leqslant k$ and moreover $\dim_{\fq}(V_{x^q,\gamma})=k$, it is sufficient to show that $V_{x^q,\gamma}\subseteq \ker(F)$.
Consider $u+u^q\gamma\in V_{x^q,\gamma}$, then
\begin{equation}\label{eq:1equationFu}
    \begin{aligned}
        F(u)= &-\delta^qu-(\gamma^{q^k}-\gamma)u^q+\gamma^q(\gamma^{q^k}-\gamma)u^{q^2}+\\
        &-\gamma^{q^2+q}(\gamma^{q^k}-\gamma)u^{q^3}+\gamma^{q^3+q^2+q}(\gamma^{q^k}-\gamma)u^{q^4}-\\
        &\vdots\\
        &+(-1)^{k-2}\gamma^{q^{k-3}+\dots+q}(\gamma^{q^k}-\gamma)u^{q^{k-2}}+\\
        &+(-1)^{k-1}\gamma^{q^{k-2}+\dots+q}(\gamma^{q^k}-\gamma)u^{q^{k-1}}+\\
        &+\delta u
    \end{aligned}
\end{equation}
and
\begin{equation}\label{eq:1equationFuq}
    \begin{aligned}
        F(u^q\gamma)= &-\delta^q\gamma u^q-(\gamma^{q^k}-\gamma)\gamma^qu^{q^2}+\\
        &+\gamma^q(\gamma^{q^k}-\gamma)\gamma^{q^2}u^{q^3}-\gamma^{q^2+q}(\gamma^{q^k}-\gamma)\gamma^{q^3}u^{q^4}+\\
        &\vdots\\
        &+(-1)^{k-3}\gamma^{q^{k-4}+\dots+q}(\gamma^{q^k}-\gamma)\gamma^{q^{k-3}}u^{q^{k-2}}\\
        &+(-1)^{k-2}\gamma^{q^{k-3}+\dots+q}(\gamma^{q^k}-\gamma)\gamma^{q^{k-2}}u^{q^{k-1}}+\\
        &+(-1)^{k-1}\gamma^{q^{k-2}+\dots+q}(\gamma^{q^k}-\gamma)\gamma^{q^{k-1}}u+\\
        &+\delta\gamma^{q^k}u^q,
    \end{aligned}
\end{equation}
since $u \in \mathbb{F}_{q^k}$.
Since $F(u+u^q\gamma)=F(u)+F(u^q\gamma)$, \eqref{eq:1equationFu} and \eqref{eq:1equationFuq} imply
\[ F(u+u^q\gamma)=(-\delta^q+(-1)^{k-1}\gamma^{q^{k-1}+q^{k-2}+\dots+q}(\gamma^{q^k}-\gamma)+\delta)u+\]
\[(-(\gamma^{q^k}-\gamma)-\delta^q\gamma+\delta\gamma^{q^k})u^q, \]
and by using that $\delta=1-(-1)^{k-2}\gamma^{q^{k-1}+\dots+1}$ we get $F(u+u^q\gamma)=0$.

\end{proof}

\begin{remark}\label{rk:xqnotrin}
Since $\gamma \notin \F_{q^k}$, then all the $q$-th powers of $x$ up to $q$-degree $k$ appear in $F$, hence $F$ is far from being a trinomial when $n>6$.
\end{remark}

\begin{corollary}\label{cor:x^qnotrin}
Let $\gamma\in\fqn$ be a root of an irreducible polynomial of degree $\frac{n}{k}\geqslant 2$ over $\fqk$.
Then $V_{x^q,\gamma}=\lbrace u+u^q\gamma:u\in\fqk\rbrace$ is not semilinearly equivalent to the examples in \cite{otal2017cyclic,santonastaso2022linearized}.
\end{corollary}
\begin{proof}
By Proposition \ref{prop:equivsubpol} and Theorem \ref{casomonomio}, we have that every subspace semilinearly equivalent to $V_{x^q,\gamma}$ cannot be the kernel of a subspace trinomial.
\end{proof}

In the next result we show that the class of examples found in Proposition \ref{Sidon caso i-j,k}, as for the monomial, cannot be the kernel of a linearized trinomial of $q$-degree $k$ and having kernel the Sidon spaces of Proposition \ref{Sidon caso i-j,k}.

\begin{proposition}
\label{nontrinomio}
Let $\gamma\in\fqn$ be a root of an irreducible polynomial of degree $\frac{n}{k}\geqslant 2$ over $\fqk$.
Let $V_{g,\gamma}=\lbrace u+g(u)\gamma:u\in\fqk\rbrace$ be an $\fq$-subspace of $\fqn$ where $g(x)= x^{q^i} + \delta x^{q^j} \in \mathcal{L}_{k,q}$ with $0<i<j<k$. If $F\in\mathcal{L}_{n,q}$ is a subspace polynomial with $\ker(F)=V_{g,\gamma}$ then $F$ is not a trinomial. 
\end{proposition}
\begin{proof}
Suppose that $F(x)=ax+bx^{q^\ell}+x^{q^k} \in \mathcal{L}_{n,q}$, with $0<\ell<k$, is a subspace polynomial with $\ker(F)=V_{g,\gamma}$, that is
\begin{equation}\label{eq:fu+gammagu=0} 
F(u+\gamma g(u))=0, \end{equation}
for any $u \in \fqk$. 
Note also that $b$ cannot be zero as $V_{g,\gamma}$ is not contained in any multiplicative coset of $\fqk$.
\eqref{eq:fu+gammagu=0} implies that
\begin{equation}\label{eq:trinzeropol} 
\begin{array}{ll}
G(u)=(a+1)u+(a\gamma +\gamma^{q^k})u^{q^i}+(a\gamma\delta +\delta \gamma^{q^k})u^{q^j}+b u^{q^\ell}+\\
b\gamma^{q^\ell}u^{q^{i+\ell}}+b\gamma^{q^\ell}\delta^{q^\ell}u^{q^{\ell+j}}=0, 
\end{array}
\end{equation}
for every $u \in \fqk$.
Hence, the left hand-side $G(u)$ of \eqref{eq:trinzeropol} can be seen as a polynomial in $u$ and \eqref{eq:trinzeropol} implies that this polynomial is the zero polynomial when is reduced modulo $u^{q^k}-u$.
To this aim, denote by 
\[\mathcal{I}=\{0,i,j\}\,\,\,\text{and}\,\,\,\mathcal{J}=\{\ell,i+\ell,j+\ell\},\]
which are seen both as subsets of $\mathbb{Z}/k\mathbb{Z}$.
We divide the proof in cases:\\
\textbf{Case 1}: $|\mathcal{I}\cap\mathcal{J}|\leq 2$.\\
In this case, we have that at least one among $\ell,i+\ell$ and $j+\ell$ is not in $\mathcal{I}$ and hence the relative coefficient in $G(u)$ is zero, that is $b=0$, a contradiction.\\
\textbf{Case 2}: $\mathcal{I}=\mathcal{J}$.\\
Note that, since $0<\ell<k$ then $i\ne \ell+i$ and $j\ne \ell+j$. Therefore, we only have two cases: 
either
\[i+\ell=0, j+\ell=i \,\,\,\text{and}\,\,\,\ell=j,\]
or 
\[i+\ell=j, j+\ell=0 \,\,\,\text{and}\,\,\, \ell=i.\]
\textbf{Case 2.1}: $i+\ell=0$, $j+\ell=i$ and $\ell=j$.\\
From \eqref{eq:trinzeropol} the coefficients of $G$ are
\begin{equation}\label{eq:systi,jcase21}
\left\{
\begin{array}{lll}
a+1+b \gamma^{q^\ell}=0,\\
a\gamma +\gamma^{q^k}+b\gamma^{q^\ell}\delta^{q^\ell}=0,\\
a\gamma\delta+\delta\gamma^{q^k}+b=0.
\end{array}
\right.
\end{equation}
Note that $a+1 \ne 0$, otherwise the first equation of the above system would imply $b=0$.
By multiplying the first equation of System \eqref{eq:systi,jcase21} by $\delta^{q^{\ell}}$ and the third one by $\gamma^{q^\ell}\delta^{q^\ell}$ we get
\[
\left\{
\begin{array}{lll}
(a+1)\delta^{q^\ell}+b \gamma^{q^\ell} \delta^{q^\ell}=0,\\
a\gamma +\gamma^{q^k}+b\gamma^{q^\ell}\delta^{q^\ell}=0,\\
a\gamma^{1+q^\ell}\delta^{1+q^\ell}+\delta^{1+q^\ell}\gamma^{q^k+q^\ell}+b\gamma^{q^\ell} \delta^{q^\ell}=0,
\end{array}
\right.
\]
which can be rewritten as follows
\[
\left\{
\begin{array}{lll}
b \gamma^{q^\ell} \delta^{q^\ell}=-(a+1)\delta^{q^\ell},\\
a\gamma +\gamma^{q^k}-(a+1)\delta^{q^\ell}=0,\\
a\gamma^{1+q^\ell}\delta^{1+q^\ell}+\delta^{1+q^\ell}\gamma^{q^k+q^\ell}-(a+1)\delta^{q^\ell}=0,
\end{array}
\right.
\]
and from the difference between the second equation multiplied by $\gamma^{q^\ell}\delta^{q^\ell+1}$ and the third equation we obtain
\[ a \gamma^{q^\ell}\delta^{2q^\ell+1}+\gamma^{q^\ell}\delta^{2q^\ell+1}=a \delta^{q^\ell}+\delta^{q^\ell}, \]
and hence
\[ \gamma^{q^\ell}\delta^{2q^\ell+1}=\delta^{q^\ell}, \]
a contradiction as $\gamma \notin \fqk$.\\
\textbf{Case 2.2}: $i+\ell=j, j+\ell=0 \,\,\,\text{and}\,\,\, \ell=i$.\\
Similar arguments to those of Case 2.1 can be performed to get again the contradiction.
\end{proof}

Arguing as in Corollary \ref{cor:x^qnotrin}, we obtain the following.

\begin{corollary}\label{cor:x^qnotrin}
Let $\gamma\in\fqn$ be a root of an irreducible polynomial of degree $\frac{n}{k}\geqslant 2$ over $\fqk$.
Let $V_{g,\gamma}=\lbrace u+g(u)\gamma:u\in\fqk\rbrace$ be an $\fq$-subspace of $\fqn$ where $g(u)= u^{q^l} + B u^{q^m}$ and $\gcd(k,m-l)=1$ with $0<l<m<k$. Then $V_{g,\gamma}$ is not semilinearly equivalent to the examples in \cite{otal2017cyclic,santonastaso2022linearized}.
\end{corollary}

\section{Conclusions and open problems}

In this paper we continued the study of Sidon spaces which are contained in the sum of two multiplicative cosets of a subfield $\fqk$ of the ambient space $\fqn$.
We give characterization results, which allowed us to check the Sidon space property in the smaller field $\fqk$ instead of the entire field $\fqn$, under the assumption that $[\fqn:\fqk]>2$.
As a consequence, we showed examples arising from scattered spaces/polynomials and others with a wide variety of algebraic behaviour. Then we conclude the paper by analyzing the natural equivalence that can be given on Sidon spaces via the connection with cyclic subspace codes. 

We point our some questions/open problems that we think might be of interest:

\begin{itemize}
    \item In Definition \ref{def:sidonpol}, we defined linearized polynomials satisfying the Sidon space property. An important task would be to investigate the exceptionality of this property,  that is to find families of linearized polynomials satisfying the Sidon space property for infinitely many extensions; see \cite{Bartoli:2020aa4} for a nice overview.
    As done for scattered polynomials, see e.g. \cite{bartoli2018exceptional,bartoli2021classification,ferraguti2021exceptional,bartoli2022towards}, a possible strategy is to translate the Sidon space property in terms of algebraic curves/function fields and then try to use heavy algebraic geometry machinery.
    \item In Corollary \ref{cor:doublepseudo}, we find an example of Sidon space with the aid of the direct sum. With the notation of Corollary \ref{cor:doublepseudo}, it is clear that if $2k \nmid m$ this example cannot be the examples described in Section \ref{sec:constructions}. When $2k \mid m$, it is not hard to show that it cannot be a Sidon space defined by a scattered polynomial, but we could not exclude that it can be the example defined by the binomial in Proposition \ref{Sidon caso i-j,k}. A possible way to show this is the following: up to equivalence and with the notation of Corollary \ref{cor:doublepseudo}, we may write the example of Corollary \ref{cor:doublepseudo} as $V_{g,\delta'}$ for some $\delta' \in \F_{q^m}$ and $g$ as in \cite[Theorem 4.20]{napolitano2022linear} and then a direct check of the equivalence with the examples in Proposition \ref{Sidon caso i-j,k} can be done (with non trivial computations).
    \item Another interesting step towards the study of Sidon spaces is to investigate those subspaces that are not contained in the sum of two multiplicative cosets of a field. Some results can be readapted but they need more discussions and we are currently working on it.
\end{itemize}

\bibliographystyle{abbrv}
\bibliography{biblio}

\noindent Chiara Castello, Olga Polverino, Paolo Santonastaso and Ferdinando Zullo,\\
Dipartimento di Matematica e Fisica,\\ 
Universit\`a degli Studi della Campania ``Luigi Vanvitelli'',\\ 
Viale Lincoln, 5,\\ 
I--\,81100 Caserta, Italy\\
E-mail: \{chiara.castello,olga.polverino, \\ paolo.santonastaso,ferdinando.zullo\}@unicampania.it\\
\\
Data sharing not applicable to this article as no datasets were generated or analysed during the current study.

\end{document}